\documentclass[12pt]{article}
\usepackage[left=3cm, right=3cm, top=2cm]{geometry}
\usepackage{amsmath, amsthm, amssymb, mathtools}
\allowdisplaybreaks
\usepackage{amscd,amssymb}
\usepackage[all]{xy}
\usepackage{color}
\usepackage{verbatim}
\usepackage{mathrsfs}
\usepackage{graphicx,epsfig}
\usepackage[T1]{fontenc}
\usepackage[ansinew]{inputenc}
\usepackage{float}
\usepackage{setspace}
\usepackage{xpatch}
\usepackage{amsfonts}
\usepackage[colorlinks]{hyperref}
\usepackage{subcaption}
\usepackage{pst-node}
\usepackage{tikz-cd}
\usetikzlibrary{positioning}
\usetikzlibrary{shapes,arrows}
\usepackage{tikz}
\usepackage{graphicx,epsfig}
\usepackage{mathrsfs}
\usepackage{verbatim}

\newtheorem{thm}{Theorem}
\newtheorem{cor}[thm]{Corollary}
\newtheorem{lem}[thm]{Lemma}
\numberwithin{thm}{section}
\newtheorem{claim}{Claim}
\usepackage{tikz-cd}
 \newtheorem{prop}[thm]{Proposition}
\theoremstyle{definition}
 \newtheorem{defn}[thm]{Definition}
 \newtheorem{rmk}{Remark}
 \newtheorem{conj}{Conjecture}

\newcommand{\bS}{\mathbb{S}}
\newcommand{\A}{\mathcal{A}}

\newcommand{\M}{\mathcal{M}}

\newcommand{\ind}{{\rm{Ind}}}

\newcommand{\s}{\mathfrak{S}}

\title{Homotopy Type of Independence Complexes of Certain Families of Graphs}

\author{Shuchita Goyal\footnote{Indian Institute of Technology Bombay, Mumbai, India. Email: shuchitagoyal23@gmail.com}, Samir Shukla\footnote{Indian Institute of Technology Bombay, Mumbai, India. Email: samirshukla43@gmail.com}, Anurag Singh\footnote{Chennai Mathematical Institute, Chennai, India. Email: asinghiitg@gmail.com}}
\begin{document}

\date{ }

\maketitle
\begin{abstract}
We show that the independence complexes of generalised Mycielskian of complete graphs are homotopy equivalent to a wedge sum of spheres, and determine the  number of copies and the dimensions of these spheres. We also prove that  the independence complexes of categorical  product of complete graphs are wedge sum of circles, upto homotopy. Further, we show that if we perturb a  graph $G$ in a certain way, then the independence complex of this new graph is homotopy equivalent to the suspension of the independence complex of $G$.
\end{abstract}



\vspace{.1in}

\hrule

\section{Introduction}

 A subset $I$ of vertex set of a graph $G$ is called {\it independent}, if the  induced subgraph of $G$ on $I$ is a collection of isolated vertices. The {\it independence complex}, $\ind(G)$, of a simple graph $G$ is the simplicial complex whose simplices are the independent sets of $G$. In last few years a lot of attention has been drawn towards the study of independence complexes of graphs.  

In \cite{BK1},  Babson and Kozlov used the topology of independence complexes of cycles to prove a conjecture by Lov\'{a}sz. Meshulam, in \cite{RM}, gave a connection  between the domination number of a graph $G$ and certain homological properties of $\ind(G)$; and their application to Hall-type theorems for coloured  independent sets.  Properties of independence complexes have also been used to study the Tverberg graphs \cite{ea} and 
the independent  system of representatives \cite{abz}.
For more on these complexes, interested reader  is referred to
\cite{Bar13, BLN, BB, csorba, eh, ea1,ea2, kk}. 

 There are only a few classes of graphs for which a closed form formula for the homotopy type of independence complexes is known. For instance, see \cite{BB} for stabe Kneser graphs, \cite{kk1} for forests, \cite{koz} for cycle graphs, and \cite{NS} for a family of regular bipartite graphs.  In this article, we compute the homotopy type of independence complex of a few families of graphs, and give the exact formula for the same. We also give results about the homotopy type of such complexes when a graph has a certain local structure. 
 


The layout of this article is as follows: In Section \ref{sec:productofcomplete}, we analyze the independence complexes  of product of complete graphs, and show that it is homotopy equivalent to a wedge of circles  (cf. Proposition \ref{thm:indkmkn}). Section \ref{sec:mycielskian} is devoted towards computation of independence complexes of generalised Mycielskian (see Definition \ref{defn:mycielskian}) of complete graphs (cf. Theorem \ref{thm:mycielskian}) which turns out to be a wedge sum of spheres. 
In Section \ref{sec:subdivision},  we show that if we perturb a  graph $G$ (locally; by removing some edges, and adding new edges and vertices) to obtain a new graph $H$, in a certain manner (refer to  Figure \ref{fig:subdivision}), then Ind$(H)$, is homotopy equivalent to the suspension of Ind$(G)$ (cf. Theorem \ref{thm:subdivision}). As an application of Theorem \ref{thm:subdivision}, we determine the homotopy type of independence complexes of cycles with certain type of subdivisions.

\section{Preliminaries}

A {\it graph} is an ordered pair $G=(V,E)$ where $V$ is called the set of vertices and $E \subseteq V \times V$, the set of unordered edges of $G$. The vertices $v_1, v_2 \in V$ are said to be adjacent, if $(v_1,v_2)\in E$. This is also denoted by $v_1 \sim v_2$, and if $v_1 = v_2$, then $v_1$ is said to be a \textit{looped} vertex. If $v$ is a vertex of $G$, then the set of its {\it neighbours} in $G$ is $\{x \in V(G): x \sim v\}$, and is denoted by $N(v)$. A graph $H$ with $V(H) \subseteq V(G)$ and $E(H) \subseteq E(G)$ is called a {\it subgraph} of the graph $G$. 
For a nonempty subset $U$ of $V(G)$, the induced subgraph $G[U]$, is the subgraph of $G$ with vertices $V(G[U]) = U$ and $E(G[U]) = \{(a, b) \in E(G) \ | \ a, b \in U\}$. In this article, $G[V(G)\setminus A]$ will be denoted by $G-A$ for $A\subsetneq V(G)$.

	The {\it complete graph} on $n$ vertices is a graph where any two distinct vertices are adjacent, and it is denoted by $K_n$. For $n \geq 3$, the {\it cycle graph $C_n$} is the graph with $V(C_n) = \{1, \ldots, n\}$ and $E(C_n) = \{(i, i+1) : 1 \leq i \leq n-1\} \cup \{(1, n)\}$.

 \begin{defn} \label{def:product}
	The {\it categorical product} of two graphs $G$ and $H$, denoted by $G\times H$ is the graph where $V(G\times H)=V(G)\times V(H)$ and  $(g,h) \sim (g',h')$ in $G\times H$, if and only if $g \sim g'$ in $G$ and $h \sim h'$ in $H$.
\end{defn}

For $r \geq 1$, let $L_r$ denote the path graph of length $r$ with loop at one end, i.e., it is a graph with  vertex set $V(L_r) = \{0, \ldots, r\}$ and  edge set $E(L_r) = \{(i,i+1) \ |  \ 0 \leq i \leq r-1\} \cup \{(0,0)\}$. 

\begin{defn} \label{defn:mycielskian}Let $G$ be  a graph and $r \geq 1$. The {\it $r$-th generalised Mycielskian}, $M_r(G)$, of $G$ is the graph $(G \times L_r) / \sim_{r}$,  where $\sim_r$ is the equivalence which identifies all those vertices whose second coordinate is $r$.  The graph $M_2(G)$ is called the {\it Mycielskian}  of $G$. 
\end{defn} 

An {\it (abstract) simplicial complex} $K$ is a collection of finite sets such that if $\tau \in K$ and $\sigma \subset \tau$, then $\sigma \in K$. The elements  of $K$ are called the  {\it simplices} of $K$.  If $\sigma \in K$ and $|\sigma |=k+1$, then $\sigma$ is said to be {\it $k$-dimensional}.  The set of $0$-dimensional simplices of $K$ is denoted by $V(K)$, and its elements are called {\it vertices} of $K$. A {\it subcomplex} of a simplicial complex $K$ is a simplicial complex whose simplices are contained in $K$. In this article, we always assume empty set as a simplex of any simplicial complex. 

The link of a vertex $v \in V(K)$ is the subcomplex of $K$ defined as 
$$
lk_K(v) : = \{\sigma \in K \ | \ v  \notin \sigma \ \text{and} \ \sigma \cup \{v\} \in K\}.
$$
The {\it star} of a simplex $\sigma \in K$ is the subcomplex of $K$ defined as
$$
st_K(\sigma):= \{\tau \in K \ |  \ \sigma \cup \tau \in K\}.
$$

	In this article, we consider any simplicial complex as a topological space, namely its  geometric realization. For the definition of geometric realization, we refer to  book \cite{dk} by  Kozlov.

  	\begin{defn}[\cite{Bar13}]
Let $K$ be a simplicial complex and $\sigma \in K$. The {\it star cluster} of $\sigma$ in $K$ is a subcomplex of $K$ defined as $$SC_K(\sigma):=\bigcup\limits_{u \in \sigma}st_K(\{u\}).$$
\end{defn}

The following results by Barmak will be used in this article.

\begin{lem}[Lemma $3.2$, \cite{Bar13}]\label{lem2.2}
The star cluster of a simplex in independence complex is contractible.
\end{lem}

\begin{lem}[Lemma $3.3$, \cite{Bar13}]\label{lem:suspension}
Let $K_1$ and $K_2$ be two contractible subcomplexes  of a simplicial complex $K$ such that $K=K_1 \cup K_2$. Then $K \simeq \ \Sigma(K_1 \cap K_2)$, where $\Sigma(X)$ denotes the suspension of space $X$.
\end{lem}

\begin{lem}[Theorem $3.6$, \cite{Bar13}]\label{lem:cluster}
    Let $G$ be  a graph and $v$ be  a non-isolated vertex of $G$ which is contained in no triangle. Then $N(v)$ is a simplex of $\ind(G)$, and
    $$\ind(G) \simeq \Sigma (st_{\ind(G)}(\{v\}) \cap SC_{\ind(G)}\big(N(v)) ).$$
\end{lem}

 The following  observation directly follows from the definition of independence complexes of graphs.
 
\begin{lem}\label{lem:union}
  Let $G$ be a graph obtained by taking disjoint union of two graphs $G_1$ and $G_2$. Then, $$\ind(G)=\ind(G_1\sqcup G_2) \simeq \ind(G_1) \ast \ind(G_2),$$ where $\ast$ denotes the join operation.
\end{lem}

Now we discuss some tools needed from discrete Morse theory (\cite{f}).

\begin{defn}[Definition $11.1$, \cite{dk}]
A {\it partial matching} on a poset $P$ is a subset $M \subseteq P \times P$ such that
\begin{itemize}
\item[(i)] $(a,b) \in M$ implies $ b \succ a;$ {\it i.e.}, $a < b$ and no $c$ satisfies $ a < c < b $, and
\item[(ii)] each $ a \in P $ belong to at most one element in $M$.
\end{itemize}

\end{defn}
   Note that, $\M$ is a  partial matching on a poset $P$ if and only if 
  there exists  $\A \subset P$ and an injective map $\mu: \A
 \rightarrow P\setminus \A$ such that $\mu(a)\succ a$ for all $a \in \A$.

 \begin{defn}
 	An {\it acyclic matching} is a partial matching  $\M$ on the poset $P$ such that there does not exist a cycle
 	\begin{eqnarray*}
 		\mu(a_1)  \succ a_1 \prec \mu( a_2) \succ a_2  \prec \mu( a_3) \succ a_3 \dots   \mu(a_t) \succ a_t  \prec \mu(a_1), t\geq 2.
 	\end{eqnarray*}
 \end{defn}

For an acyclic partial matching on $P$, those elements of $P$ which do not belong to the matching are said to be 
{\it critical }.

\begin{thm}[Theorem $11.13$, \cite{dk}]\label{acyc3}
Let $\Delta$ be a simplicial complex and $M$ be an acyclic matching on the face poset of $\Delta$. Let $c_i$ denote the number of critical $i$-dimensional cells of $\Delta$ with respect to the matching $M$. Then $\Delta$ is homotopy equivalent to a cell complex $\Delta_c$ with $c_i$ cells of dimension $i$ for each $i \geq 0$, plus a single $0$-dimensional cell in the case where the empty set is also paired in the matching.
\end{thm}

Following can be inferred from Theorem \ref{acyc3}.

\begin{rmk} \label{acyc4}
If an acyclic matching has critical cells only in a fixed dimension $i$, then $\Delta$ is homotopy equivalent to a wedge of $i$-dimensional spheres.
\end{rmk}

  \section{Independence complex of \texorpdfstring{$K_m \times K_n$}{categorical product of complete graphs}}\label{sec:productofcomplete}
In this section, we compute the independence complex of $K_m \times K_n$ for $m,n \geq 2$.
We first  start by defining an acyclic matching on the face poset of a general simplicial complex; and then use a special case of this matching to prove the result for $\ind(K_m \times K_n)$.

  Let $K$ be a simplicial complex and let $X = \{x_1, \ldots, x_n\}\subseteq V(K)$. The elements of $X$ are ordered as; $
  x_1 < x_2 < \ldots <  x_n.
  $
  
  Let $P$ be the face poset of $(K, \subseteq)$.
  For  $1 \leq i \leq n ,$ define 
  \begin{eqnarray*}
  &&A_{x_i} = \{\sigma \in A_{x_{i-1}}' \ | \ x_i \ \notin \sigma , \ \text{and} \ \sigma \, \cup \{x_i \} \ \in A_{x_{i-1}}'\}, ~\text{where}~ A_{x_0}' =P, \label{am:1}\\
  &&\mu_{x_i} : A_{x_i}  \rightarrow A_{x_{i-1}}' \setminus A_{x_i} ~\text{by} ~  \mu_{x_i}(\sigma) = \sigma \, \cup \{x_i\}~~\text{and}\label{am:2}\\
  &&A_{x_i}'= A_{x_{i-1}}' \setminus \s_{x_i}, ~\text{where}~ \s_{x_i} = A_{x_i}\cup \mu_{x_i} (A_{x_i}).\label{am:3}
  \end{eqnarray*}
  
  We note that by construction, $A_{x_i} \cap A_{x_j} =\emptyset$ whenever $ i \neq j$.
  Let $A = \bigcup \limits_{i=1}^{n} A_{x_i}$ and $\mu_K^{X}: A \rightarrow P \setminus A$ be defined by $\mu_K^X(\sigma) =\mu_{x_i} (\sigma)$, where $x_i$ is the unique element such that $\sigma \in A_{x_i}$.

  Clearly, $\mu_K^{X}$ is  injective  and is therefore, a well defined partial matching on $P$.  It follows from \cite[Proposition 3.1]{NS} that $\mu_K^X$ is an acyclic matching. For the sake of completeness, we give a proof here as well.

  \begin{prop}\label{prop:acyclic}
  	$\mu_K^{X}$ is an acyclic matching on $P$.
  \end{prop}
  \begin{proof}
  	Let there exist distinct cells $\sigma_1, \ldots, \sigma_t \in A$ such that $\mu_K^{X}(\sigma_i) \succ 
  	\sigma_{i+1 \ (\text{mod} \ t)}, 1 \leq i \leq t$.
  	
  	Let $x \in X$ be the least element such that $\{\sigma_1, \ldots, \sigma_t\} \cap A_x \neq \emptyset$.
  	Without loss of generality, assume that $\sigma_1 \in A_x$, {\it i.e.}, $x \notin \sigma_1$ and $\mu_K^X(\sigma_1) = \sigma_{1} \cup \{x\}$. $\mu_K^X(\sigma_1) \succ \sigma_2$ and  $\sigma_1 \neq \sigma_2$ implies that there exists $x' \in \mu_1(\sigma_1), x' \neq x$ such that 
  	$\sigma_2 = \mu_K^{X}(\sigma_1) \setminus \{x'\}$. We now have the following two possibilities:
  	\begin{enumerate}
  		\item $x \in \sigma_{t}$. 
  		
  		$\sigma_1 \in A_x$ implies that $x \notin \sigma_1$.  $x \in \sigma_t$ implies that $x \in \mu_K^X(\sigma_t)$. Therefore, $\sigma_1 = \mu_K^X(\sigma_t) \setminus \{x\}$ which implies that $\mu_K^X(\sigma_1) = \mu(\sigma_t)$ a contradiction, since $\sigma_1 \neq \sigma_t$.
  		
  		\item $x \notin \sigma_{t}$,  {\it i.e.}, there exists  a least $l \in \{2, \ldots, t\}$ such that $x \notin \sigma_l$.
  		
  		$x \in \mu_K^X(\sigma_{l-1})$ and $x \notin \sigma_l$ implies that $\sigma_l = \mu_K^X(\sigma_{l-1}) \setminus \{x\}$ {\it i.e.}, $\mu_K^X(\sigma_{l-1}) = \sigma_l \cup \{x\}$. Since $\sigma_l$ and $\mu_K^X(\sigma_{l-1}) \notin A_i \cup \mu_{x_i}(x_i) \ \forall \ i < x$, from the definition $\sigma_l \in A_x$. This implies that $\mu_K^X(\sigma_l) = \sigma_l \cup \{x\} = \mu_K^X(\sigma_{l-1})$, which implies that $\sigma_l = \sigma_{l-1}$, a contradiction.
  	\end{enumerate}
  	Therefore,  $\mu_K^{X}$ is an acyclic matching on $P$.
  \end{proof}
 
  Let $m,n \geq 2$ and $V(K_m)=\{a_1,\ldots,a_m\}$, $V(K_n)= \{b_1,\ldots,b_n\}$. 
  
  \begin{rmk}\label{rmk:maximalsets}
  Observe that the maximal simplices  of $\ind(K_m\times K_n)$ are only of  the following two types:
 
  \begin{enumerate}
 \item  sets of the form $\{(a_i,b_j) \ | \ j \in [n]\}$, where $i \in [m]$, and
 \item  sets of the form $\{(a_i,b_j) \ | \ i \in [m]\}$, where $j \in [n]$.
 \end{enumerate}
 \end{rmk}
 
  Using the above classification of simplices of $\ind(K_m\times K_n)$, we  prove the following result.

\begin{prop}\label{thm:indkmkn}
	Let $m,n \geq 2$. Then
	$$\emph{\ind}(K_m \times K_n) \simeq \bigvee\limits_{(m-1)(n-1)} \bS^1.$$
\end{prop}
\begin{proof}
Let $I : = \ind(K_m \times K_n)$ and let $J=\{(a_1, b_i) ~|~  1 \leq i \leq n\}  \cup \{(a_i, b_1) ~| \ 2 \leq i \leq m\} \subseteq V(I)$. Further, let $P_{m,n}$ be the face poset of $(I, \subseteq)$.  We define the ordering on $J$ as follows:
\begin{equation*}
	(a_1, b_1) < \ldots < (a_1, b_n) < (a_2, b_1) < {(a_3, b_1)} < \ldots  < (a_m, b_1).
\end{equation*}

Let $\mu_{I}^{J}$ be the matching defined as in the beginning of this section with respect to the ordering of elements of $J$ given as above. From Proposition \ref{prop:acyclic} $\mu_{I}^J$ is an acyclic matching. Let $C$ be the set of critical cells for the  matching $\mu_I^{J}$.

\begin{claim}\label{claim:criticalcell}

C= $\{\{(a_i, b_1), (a_i, b_j)\} \  | \   2 \leq i \leq m, 2 \leq  j \leq n \}$.
\end{claim}

\begin{proof}[Proof of Claim \ref{claim:criticalcell}] In this proof, for the convenience of notation,  we  denote  $\mu_I^J$ by $\mu$. 

Here, we first show that every element of $C$ is critical. Let $i \in \{2, \ldots, m\}$ and $j \in \{2, \ldots, n\}$. First observe that $\mu(\{(a_i, b_j)\}) = \{(a_1, b_j), (a_i, b_j)\}$. Since $i, j \geq 2$, it follows from the definition of $\mu$ that $\{(a_i,b_1),(a_i,b_j)\}$ is a critical cell.

Now, let $\sigma \in I$ be a critical cell. Note that $\mu(\{\emptyset\}) = \{(a_1, b_1)\}$,  therefore $\sigma \neq \{(a_1, b_1)\}$.  Since for each $j \geq 2,\  \mu(\{(a_1, b_j)\}) = \{(a_1, b_j),(a_1, b_1) \}$; and for each $i \geq 2$ and $k \geq 1$, $\mu(\{(a_i, b_k)\} ) = \{(a_i, b_k), (a_1, b_k)\}$, we thus conclude that $\sigma$ has at least two elements. 
From Remark \ref{rmk:maximalsets}, either  $\sigma = \{(a_{i_1}, b_{j}), \ldots,(a_{i_t}, b_{j})\}$ for some fixed $j  \in [n]$ and $t \geq 2$ or 
$\sigma = \{(a_i, b_{j_1}), \ldots,(a_i, b_{j_l})\}$ for some fixed $i \in [m]$ and $l \geq 2$.

Suppose  $\sigma = \{(a_{i_1}, b_{j}), \ldots,(a_{i_t}, b_{j})\}$ for some $j  \in [n]$ and $t \geq 2$. If  $(a_1, b_j) \notin \sigma $, then $\mu(\sigma) = \sigma \cup \{(a_1, b_j)\}$; and if $(a_1, b_j)  \in \sigma $, then $\sigma = \mu(\sigma \setminus \{(a_1, b_j)\})$, which  contradicts that $\sigma$ is a critical cell. Therefore, $\sigma = \{(a_i, b_{j_1}), \ldots,(a_i, b_{j_l})\}$ for some $i \in [m]$ and $l \geq 2$.

Note that, if $(a_i, b_1) \notin \sigma$ then $\mu(\sigma) = \sigma \cup \{(a_1, b_1)\}$, which is again a contradiction. Therefore, $(a_i, b_1) \in \sigma$. Further, if $i=1$, then $\sigma = \mu(\sigma \setminus \{(a_1, b_1)\})$. Therefore, $\sigma = \{(a_i, b_{j_1}), \ldots,(a_i, b_{j_l})\}$ for some $i \in \{2,\ldots,m\}$, $l \geq 2$ and $(a_i, b_1) \in \sigma$.

To prove Claim \ref{claim:criticalcell}, it now suffices to show that $|\sigma|=2$. Suppose $|\sigma| \geq 3$. Since $|\sigma \setminus \{(a_i, b_1)\}| \geq 2$ and $i\geq 2$, by definition of $\mu$, we have $\mu(\sigma \setminus \{(a_i, b_1)\})=\sigma$, which is a contradiction to the fact that $\sigma$ is critical and therefore result follows.
\end{proof}

 From Claim \ref{claim:criticalcell}, all the critical cells for matching $\mu$ are of the same dimension, {\it i.e.}, one dimensional. Moreover, the cardinality of the set $C$ is $(m-1)(n-1)$. Therefore result follows from Remark \ref{acyc4}.
\end{proof}


\begin{rmk} Observe that  the graph $\underbrace{K_2 \times \ldots \times K_2}_{(r-1)\text{-copies}} \times K_n$ is isomorphic to $2^{r-2}$ disjoint copies of $K_2 \times K_n$. Therefore, using Lemma \ref{lem:union}, we get

$$\ind(\underbrace{K_2 \times \ldots \times K_2}_{(r-1)\text{-copies}} \times K_n) \simeq \bigvee\limits_{(n-1)^{2^{r-2}}} \mathbb{S}^{2^{r-1}-1}.$$ 

It is thus natural to ask if one can generalise the Proposition \ref{thm:indkmkn} to $r$-fold product of complete graphs for $r\geq 3$, {\it i.e.}, if the independence complexes of $r$-fold product of complete graphs are homotopy equivalent to wedge sum of spheres. 
We strongly believe that the independence complexes of $r$-fold product of complete graphs are homotopy equivalent to wedge of spheres of dimension $2^{r-1}-1$.
\end{rmk}

\noindent \begin{minipage}[m]{.35\textwidth}
    {In support of our intuition, we present  our computer based computations for the Betti numbers, denoted $\beta_i$, of the independence complexes of $K_2 \times K_3 \times K_n$ in Table \ref{table:3}.}
\end{minipage}
\begin{minipage}[m]{0.05\textwidth}
{ ~}
\end{minipage}
\begin{minipage}[m]{.6\textwidth}
{\begin{table}[H]
	\centering
	\begin{tabular}{ |c|c|c|c|c| }
		\hline 
		{\bf $n$}  & {\bf $\beta_3$} & {\bf $\beta_i$, $i \neq 3$} \\ 
		\hline \hline
		2& $4$ & $0$ \\ 
		\hline
		3& $14$ & $0$ \\ 
		\hline
		$4$ & $30$ &  $0$ \\ 
		\hline
		$5$ &  $52$&  $0$\\ 
		\hline
			$6$ & $80$ &  $0$\\ 
		\hline
	\end{tabular}
	\caption{Betti numbers of  {\bf $\ind(K_2 \times K_3 \times K_n$)}}
	\label{table:3}
\end{table}}
\end{minipage}

\pagebreak
Based on our computations, we propose the following conjecture 

\begin{conj} For $n \geq 2$, $$\ind(K_2 \times K_3 \times K_n)\simeq \bigvee\limits_{(n-1)(3n-2)} \mathbb{S}^3.$$
	\end{conj}

  \section{Independence complex of \texorpdfstring{$M_r(K_n)$}{}}\label{sec:mycielskian}
 
 This section is devoted to the computation of independence complexes of Mycielskian of graphs. To start with, we compute  $\ind(M_2(G))$ for any graph $G$. We then focus on the generalised Mycielskian of graphs, and determine the homotopy type of \ind$(M_r(K_n))$ for any $n$ and $r\ge 2$.

 \begin{thm} \label{thm:requal2}
 	For any graph $G$, 
 	\ind$(M_2(G)) \simeq \Sigma (\ind(G))$. 
 	
 \end{thm}
\begin{proof}

Let $V(G) = \{v_1, \ldots, v_n\}$ and let $w = (v_1,2) = \ldots = (v_n, 2)$. Let $K= st_{\ind(M_2(G))}(w) \cap SC_{\ind(M_2(G))}(\{(v_1,1), \ldots, (v_n,1)\})$.  Since $N(w)=\{(v_1,1), \ldots, (v_n,1)\}\in \ind(M_2(G))$, Lemma \ref{lem:cluster} implies that 
  	\begin{center}
  $	\ind(M_2(G)) \simeq \Sigma (K).$
  \end{center}
  	
  Let $H$ be the subgraph of $M_2(G)$ induced by $\{(v_1, 0), \ldots, (v_n, 0)\}$. Clearly, $H \cong G$ and therefore $\ind(H) \cong \ind(G)$. We now show that $K = \ind(H)$.
    
   Let $\sigma \in \ind(H)$ and $(v_i,0) \in \sigma$, then $N((v_i,0)) \cap \sigma = \emptyset$. Since $N((v_i,1)) \subseteq N((v_i,0)) \cup \{w\}$,  $N((v_i,1)) \cap \sigma = \emptyset$ thereby implying that $\sigma \in $ $ st_{\ind(M_2(G))}(\{(v_i,1)\})$. Since $\sigma \subseteq V(H)$, we see that $\sigma \cup \{w\} \in \ind(M_2(G))$. Therefore, 
  	$\sigma \in st_{\ind(M_2(G))}(\{w\}) \cap SC_{\ind(M_2(G))}(\{(v_1,1), \ldots, (v_n,1)\})$ and hence $\ind(H) \subseteq K$.
  	
  	Now suppose that $\sigma \in K$. For each $i$, $w$ is adjacent to  $(v_i, 1)$ in $M_2(G)$, therefore $\sigma \cap \{w, (v_i,1)\} = \emptyset $ for all $i$, and hence $K \subseteq \ind(H)$.
  \end{proof}

    We now fix some notations and list a few  results which will be used in this section for the computation of the independence complex of the generalised Mycielskian of complete graphs. 
    
    For a vertex $v $ of a graph $G$, $N[v] := N(v) \cup \{v\}$. Also, if $A \subseteq V(G)$, then $N(A) :=\bigcup\limits_{v \in A} N(v)$ and $N[A] := \bigcup\limits_{v \in A} N[v]$.

\begin{lem}[Lemma $3.4$, \cite{ea1}]\label{lem:folding} 
  Let $G$ be a graph and $u, u^\prime \in V(G)$ such that $N(u)\subseteq N(u^\prime)$. Then, $$\ind(G) \simeq \ind(G\setminus u^\prime).$$
 \end{lem}

  \begin{lem}[Proposition $2.10$, \cite{JS10}]\label{lem:addedge} Let $G$ be  a graph and let $\{a,b\} \in \ind(G)$. If $\ind(G - N[\{a,b\}])$ is collapsible, then 
  	$\ind(G)$ collapses onto $\ind(\tilde{G})$, where $V(\tilde{G}) = V(G)$ and $E(\tilde{G}) = E(G) \cup \{(a,b)\}$. In particular, $\ind(G) \simeq \ind(\tilde{G})$.
  	\end{lem}
 
  \begin{lem}[Lemma $2.1$, \cite{kk}]\label{lem:simplicial}
  	Let $G$ be  graph and $v$ be a simplicial vertex\footnote{A vertex $v$ of $G$ is called {\it simplicial} if the subgraph induced by $N(v)$ is a complete graph.} of $G$. Let $N(v) = \{w_1, w_2, \ldots, w_k\}$. Then 
  	$$
  	\ind(G) \simeq \bigvee\limits_{i= 1}^{k} \Sigma(\ind(G - N[	w_i])).
  	$$
  	\end{lem}

  \begin{defn}
       Let $p: X \to Y$ and $q: X \to Z$ be two continuous maps. The  {\it pushout} of the diagram $Y \xleftarrow{p}  X \xrightarrow{q} Z$ is the space 
       $$\Big(Y \bigsqcup   Z \Big)/ \sim,$$ where $\sim$ denotes  the equivalence relation $p(x) \sim q(x)$ for $x \in X$.

   The {\it homotopy  pushout} of  $Y \xleftarrow{p}  X \xrightarrow{q} Z$ 
  is the space $\big(Y \sqcup ( X \times I ) \sqcup  Z\big)/ \sim$, where 
  $\sim$ denotes  the equivalence relation $(x,0) \sim p(x)$, and $(x,1) \sim q(x)$ for $x \in X$. It can be shown that homotopy pushouts of any two homotopy equivalent diagrams are homotopy equivalent.

  \end{defn}
  
  \begin{rmk}\label{rmk:homotopypushout}
      If spaces are CW complexes and maps are subcomplex  inclusions, then their homotopy pushout and pushout spaces are equivalent up to homotopy. For elaborate discussion of these results, we refer interested readers to \cite[Chapter 7]{bredon}.
  \end{rmk}

  \begin{lem} \label{lem:link} Let $X$ be  a simplicial complex and $v\in V(X)$. Let $Y = \{\sigma \in X \ | \ v \notin \sigma \}$ be  a subcomplex of $X$. If  $lk_X(v)$ is contractible, then $X \simeq Y$. 
\end{lem}
\begin{proof}
Let $A= lk_X(v)$ and let $Z$ be the homotopy pushout of the diagram $A \xleftarrow{=}  A \xhookrightarrow{} Y$.  Since $A \times I$ is homotopy equivalent to $A$, $Y \simeq Z$. Also, contractibility of $A$ implies that $Z$ is of the same homotopy type as $Z/ (A \times \{1\})$. Therefore,  
	$	Y \simeq Z \simeq Z / (A \times \{1\})$ which is homeomorphic to $X$.
\end{proof}

  \begin{lem}\label{lem:contractible}
  	Let $n \geq 2$ and $X_1, X_2, \ldots, X_n$ be simplicial complexes. If each $X_i$ is contractible and for each $j \in \{2,3,\ldots, n\}$, $\big{(} \bigcup\limits_{i=1}^{j-1}X_i\big{)} \cap X_j \simeq \bigvee\limits_{r} \bS^{k}$, then $X_1 \cup X_2 \cup  \ldots \cup X_n \simeq \bigvee\limits_{(n-1)r} \bS^{k+1} $.
  \end{lem}
  
  \begin{proof}Observe that $X_1 \cup X_2$ is the pushout of the diagram $X_1\xhookleftarrow{} X_1 \cap X_2 \xhookrightarrow{} X_2$, where $\hookrightarrow$ denotes inclusion maps. From Remark \ref{rmk:homotopypushout}, the homotopy pushout and pushout of $X_1\xhookleftarrow{} X_1 \cap X_2 \xhookrightarrow{} X_2$ are homotopy equivalent to each other. Further, the homotopy pushout of $X_1\xhookleftarrow{} X_1 \cap X_2 \xhookrightarrow{} X_2$ is homotopy equivalent to the homotopy pushout of $\{ \text{point}\}\longleftarrow{} \bigvee\limits_{r} \bS^{k} \longrightarrow{} \{\text{point}\}$ (since $X_1$ and $X_2$ are contractible and $X_1 \cap X_2 \simeq \bigvee\limits_{r} \bS^{k}$). Moreover, homotopy pushout of $\{\text{point}\}\longleftarrow{} \bigvee\limits_{r} \bS^{k} \longrightarrow{} \{\text{point}\}$ is homotopy equivalent to $\Sigma\big{(}\bigvee\limits_{r} \bS^{k}\big{)}$. Therefore, $X_1 \cup X_2 \simeq \bigvee\limits_{r} \bS^{k+1}$.
  
 Let $n \geq 3$. Inductively assume that for any $2 \leq t < n$, $\bigcup\limits_{i=1}^{t}X_i \simeq \bigvee\limits_{(t-1)r} \bS^{k+1}$. 
  In particular, $\bigcup\limits_{i=1}^{n-1}X_i \simeq \bigvee\limits_{(n-2)r} \bS^{k+1}$.
  Further, the pushout of the diagram $\bigcup\limits_{i=1}^{n-1}X_i \xhookleftarrow{} \big{(}\bigcup\limits_{i=1}^{n-1}X_i\big{)} \cap X_n \xhookrightarrow{} X_n$ is the space $\bigcup\limits_{i=1}^{n}X_i$. Thus, from Remark \ref{rmk:homotopypushout}, $\bigcup\limits_{i=1}^{n}X_i$ is homotopy equivalent to the homotopy pushout of the diagram $\bigvee\limits_{(n-2)r} \bS^{k+1} \longleftarrow{} \bigvee\limits_{r} \bS^{k}  \longrightarrow{} \{\text{point}\}$ which is homotopy equivalent to $\bigvee\limits_{(n-2)r+r} \bS^{k+1}$.

  \end{proof}

\begin{lem} \label{lem:wedge}
  	Let $n \geq 2$ and $X_1, X_2, \ldots, X_n$ be simplicial complexes. If for each $i \in \{1,2,\ldots,n\}$, $X_i\simeq \bigvee\limits_{r}\bS^k$ and for each $j \in \{2,3,\ldots, n\}$, $\big{(} \bigcup\limits_{i=1}^{j-1}X_i\big{)} \cap X_j$ is contractible, then $X_1 \cup X_2 \cup  \ldots \cup X_n\simeq \bigvee\limits_{nr}\bS^k$.
  	\end{lem}
  \begin{proof}
  
Using similar arguments as in the proof of Lemma \ref{lem:contractible}, we get that $X_1 \cup X_2$ is homotopy equivalent to the homotopy pushout of $\bigvee\limits_{r}\bS^k \xhookleftarrow{} \{\text{point}\}\xhookrightarrow{} \bigvee\limits_{r}\bS^k$ (since $X_1\simeq \bigvee\limits_{r}\bS^k \simeq X_2 $ and $X_1 \cap X_2$ is contractible).

Further, homotopy pushout of $\bigvee\limits_{r}\bS^k \xhookleftarrow{} \{\text{point}\}\xhookrightarrow{} \bigvee\limits_{r}\bS^k$ is homotopy equivalent to $\bigvee\limits_{r+r} \bS^{k}$. Thus, $X_1 \cup X_2 \simeq \bigvee\limits_{2r} \bS^{k}$. As before, the result now follows from induction.
 \end{proof}

\begin{prop}\label{prop:main} Let $r \geq 0$ and $n \geq 2$. Then
			\begin{center}
		$ \ind(K_n\times L_{r}) \simeq 
		\begin{cases}
		\bigvee\limits_{(n-1)^{k+1}} \bS^{2k} & \text{if } r=3k, \\
		\hspace*{0.3cm} \{\rm{point}\} & \text{if} \ r = 3k+1,\\
		
		\bigvee\limits_{(n-1)^{k+1}} \bS^{2k+1} & \text{if } r= 3k+2. \\
		\end{cases}$
	\end{center}

	\end{prop}

\begin{proof}Let $r=3k+t$ for some $t \in \{0,1,2\}$ and $k\geq 0$. We  prove this result by induction on $k$.

To prove the base step, let $k=0$. We show that the result holds for $t \in \{0,1,2\}$. If $t=0$, then  $K_n \times L_0$ isomorphic to $K_n$ implies $$\ind(K_n\times L_0) \cong \ind({K_n}) = \bigvee\limits_{n-1} \bS^0.$$

For $t=1$, let $H_1$ be the induced subgraph of $K_n\times L_1$ with vertex set $\{(i,1)\ | \ 1\leq i \leq n\}$. Since $H_1$ does not have any edge, $\ind(H_1)\simeq \{\text{point}\}$. Observe that, in $K_n\times L_1$, $N((i,1))\subseteq N((i, 0))$ for each $i \in \{1,2, \ldots, n\}$. Now repeated use of Lemma \ref{lem:folding} for each $i$ gives us
$\ind(K_n\times L_1)  \simeq \ind(H_1) \simeq \{\text{point}\}.$
This proves the result for $k=0$ and $t = 1$.

 Finally, if $t=2$, let $H_2$ be the induced subgraph of $K_n \times L_2$ with vertex set $\{(i,j)\ | \ 1\leq i \leq n, \ 1\leq j \leq 2)\}$. Clearly, $H_2 \cong K_2 \times K_n$. We observe that in $K_n \times L_2$, $N((i,2))\subseteq N((i,0))$ for each $i \in \{1,2, \ldots, n\}$. By Lemma \ref{lem:folding} and Proposition \ref{thm:indkmkn}, we conclude that

$$\ind(K_n\times L_2) \simeq \ind(K_2 \times K_n)\simeq \bigvee\limits_{n-1} \bS^1.$$

Inductively assume that the result is true for $k < s$ and $t \in \{0,1,2\}$. We now  show that the result holds for $k=s >0$ and every $t \in \{0,1,2\}$.

Let $H$ be the induced subgraph of $K_n \times L_{3s+t}$ with vertex set $V(K_n \times L_{3s+t}) \setminus \{(i,3s+t-2)\ | \ 1\leq i \leq n\}$. Further, in $K_n\times L_{3s+t}$, $N((i,3s+t))\subseteq N((i,3s+t-2))$. Thus, using Lemma \ref{lem:folding}, we have that $\ind(K_n\times L_{3s+t})\simeq \ind(H)$. Now observe that, $H \cong (K_2\times K_n)\bigsqcup (K_n \times L_{3s+t-3})$. Using Lemma \ref{lem:union} and Proposition \ref{thm:indkmkn}, we get 

\begin{align*}
\ind(K_n\times L_{3s+t}) & \simeq \ind(H)\\
& \simeq \ind(K_2\times K_n) \ast \ind(K_n \times L_{3s+t-3})\\
& \simeq \big{(} \bigvee\limits_{n-1} \bS^1\big{)}\ast \big{(}\ind(K_n \times L_{3(s-1)+t})\big{)}.
\end{align*}

By  induction hypothesis, we get
\begin{align*}
 \ind(K_n\times L_{3s+t}) & \simeq \begin{cases}		
        \big{(} \bigvee\limits_{n-1} \bS^1\big{)}\ast \big{(}\bigvee\limits_{(n-1)^{s}} \bS^{2(s-1)}\big{)} & \text{if } t=0, \\
		\big{(} \bigvee\limits_{n-1} \bS^1\big{)}\ast \{\text{point}\} & \text{if} \ t= 1,\\
		\big{(} \bigvee\limits_{n-1} \bS^1\big{)}\ast\big{(}\bigvee\limits_{(n-1)^{s}} \bS^{2(s-1)+1}\big{)} & \text{if } t=2. \\ 
 \end{cases}\\
 & \simeq \begin{cases}		
        \bigvee\limits_{(n-1)^{s+1}} \bS^{2s} & \text{if } t=0, \\
	   \hspace*{0.3cm} \{\text{point}\} & \text{if} \ t= 1,\\
		\bigvee\limits_{(n-1)^{s+1}} \bS^{2s+1} & \text{if } t=2. \\ 
 \end{cases}
\end{align*}
This completes the proof of Proposition \ref{prop:main}.
	\end{proof}

 We fix a natural number $n \geq 3$, and define a few notations that would be used in rest of this section. 
 
\begin{enumerate}
    \item Let $V(K_n) = \{1,2,\dots, n\}$.
    \item In $M_r(K_n)$, the equivalence class of vertices with second coordinate as $r$  is denoted by $w_r$.
    \item For $r \geq 3$, let $i \in \{1, \ldots, n\}$ and $j \in \{1, \ldots, r-1\}$. We define $I_{i, j}^n$ to be the subgraph of $M_r(K_n)$ induced by  $V(M_{j}(K_n)) \setminus \{w_{j}\} \cup \{(i,j)\}$. Also, let $I_{i,0}^n$ be the subgraph of $M_r(K_n)$ induced by the vertex $(i,0)$.  
      \end{enumerate}

\begin{figure} [H]
	\centering
	\resizebox{0.3\textwidth}{!}{
		\begin{tikzpicture}
		[scale=0.7, vertices/.style={minimum size=0.7pt, draw, fill=black, circle, inner
			sep=0.7pt}]
		\node[vertices, label=left:{$(1,0)$}] (a) at (1,2.5) {};
		\node[vertices, label=right:{$(2,0)$}] (b) at (5,2.5) {};
		\node[vertices,label=below:{$(3,0)$}] (c) at (3,1) {};
			\node[vertices, label=left:{$(1,1)$}] (d) at (1,5) {};
		\node[vertices, label=right:{$(2,1)$}] (e) at (5,5) {};
		\node[vertices,label=above:{$(3,1)$}] (f) at (3,5) {};
		\node[vertices,label=above:{$(3,2)$}] (g) at (3,7) {};
		\foreach \to/\from in
		{a/b, b/c, c/a, d/b, d/c, e/a, e/c, f/a, f/b, g/d, g/e} \draw [-] (\to)--(\from);
		\end{tikzpicture}}
		\caption{$I_{3,2}^3$} \label{fig:I32}
	\end{figure}	
 
{\bf Example:} $I_{3,2}^3$ is the induced subgraph of $M_r(K_3)$ on the vertex set $\{(s,t) \ : \ 1\leq s \leq 3, \ 0 \leq t \leq 1\} \cup \{(3,2)\}$ (see Figure \ref{fig:I32}).

Since $n\geq 3$ is fixed for the rest of this section, we would write $I_{i,j}$ to denote $I_{i,j}^n$ for the simplicity of notation.

\begin{lem}  \label{lem:base}Let $I_{1,t}$ be as defined above, then 

\begin{align*}
 \ind(I_{1,t}) & \simeq \begin{cases}		
       \bigvee\limits_{n-1} \bS^{0} & \text{if} \ t=1, \\
		 \bigvee\limits_{n-1} \bS^{1} & \text{if} \ t= 2,\\
		 \{\rm{point}\} & \text{if}\  t=3. \\ 
 \end{cases}
 \end{align*}
	\end{lem}

\begin{proof}
	In $I_{1,1}$, $N((1,0)) = \{(2,0), \ldots, (n,0)\} = N((1,1))$ and therefore  by Lemma \ref{lem:folding}, $\ind(I_{1,1}) \simeq \ind(I_{1,1} \setminus \{(1,1)\}) \cong \ind(K_n) \simeq \bigvee\limits_{n-1} \bS^{0}$.
	
	Recall that a vertex $v$ is simplicial if the subgraph induced by $N(v)$ is a complete graph.	In $I_{1,2}$, $N((1,1)) = \{(2,0), \ldots, (n,0)\}$ and therefore $(1,1)$ is a simplicial vertex. Moreover, for each $2 \leq i \leq n$, $V(I_{1,2}) \setminus N[(i,0)] = \{(1,2), (i,1)\}$, implying that $I_{1,2} - N[(i,0)] \cong K_2$. Therefore, using Lemma \ref{lem:simplicial}, we get that
	$$
	\ind(I_{1,2}) \simeq \bigvee\limits_{n-1} \Sigma(\ind(K_2)) \simeq \bigvee\limits_{n-1} \Sigma(\bS^{0}) \simeq \bigvee\limits_{n-1} \bS^{1}.
	$$
	
	Since the graph $I_{1,3} - N[\{(1,1), (2,1)\}]$ contains an isolated vertex $(1,3)$, $\ind(I_{1,3} - N[\{(1,1), (2,1)\}])$ is a cone by Lemma \ref{lem:union}, and hence collapsible. 
	Using Lemma \ref{lem:addedge}, $\ind(I_{1,3}) \simeq \ind(I_{1,3}')$, where $I_{1,3}'$ is the graph with $V(I_{1,3}') = V(I_{1,3})$ and $E(I_{1,3}') = E(I_{1,3}) \cup \{((1,1), (2,1))\}$.
	We repeat this process for all pair of vertices
	$((i,1), (j,1))$ $ 1\leq i \neq  j \leq n $ and apply Lemma \ref{lem:addedge}, which thereby implies that
	$\ind(I_{1,3}) \simeq \ind(\tilde{I}_{1,3})$, where $V(\tilde{I}_{1,3}) = V(I_{1,3})$ and $E(\tilde{I}_{1,3}) = E(I_{1,3}) \cup \{((i,1),(j,1)) \ : \ 1 \leq i \neq j \leq n\}$. For each $1 \leq i \leq n$, $\tilde{I}_{1,3} - N[\{(i,1)\}]$ contains an isolated vertex $(i,0)$ and therefore $\ind(\tilde{I}_{1,3} - N[\{(i,1)\}])$  is collapsible by Lemma \ref{lem:union}. Now, using the fact that $(1,2)$ is a simplicial vertex in $\tilde{I}_{1,3}$ and $\ind(\tilde{I}_{1,3} - N[\{(i,1)\}]) \cong \ind(\tilde{I}_{1,3} - N[\{(j,1)\}])$ for all $2 \leq i \neq j \leq n$, by Lemma \ref{lem:simplicial}, $\ind(\tilde{I}_{1,3})$ is contractible. Hence, $\ind(I_{1,3})$ is contractible.
	\end{proof}

We now generalise Lemma \ref{lem:base} and compute the homotopy type of independence complex of $I_{1,t}$, for any natural number $t$.

\begin{lem} \label{lem:I_{i,r-2}}Let $t \geq 6$.  Then

\begin{align*}
 \emph{\ind}(I_{1,t-2}) & \simeq \begin{cases}		
       \bigvee\limits_{(n-1)^{k}} \bS^{2(k-1)+1} & \text{if} \ t=3k+1, \\
		  \{\rm{point}\} &  \text{if} \ t=3k+ 2,\\
		\bigvee\limits_{(n-1)^{k}} \bS^{2(k-1)}& \text{if}\  t=3k. \\ 
 \end{cases}
 \end{align*}
\end{lem}

\begin{proof}
	To prove this, we first construct a graph $\tilde{I}_{1,t-2}$ that contains $I_{1,t-2}$ as a subgraph such that $(1,t-3)$ is  a simplicial vertex in $\tilde{I}_{1,t-2}$ and $\ind(I_{1,t-2}) \simeq \ind(\tilde{I}_{1,t-2})$. Observe that the vertex $(1,t-2)$ is an isolated vertex in $I_{1, t-2} - N[\{(1,t-4), (2,t-4)\}]$ and hence $\ind(I_{1,t-2} - N[\{(1,t-4),(2,t-4)\}])$ is collapsible. By Lemma \ref{lem:addedge}, $\ind(I_{1,t-2}) \simeq \ind(H)$, where $V(H) = V(I_{1, t-2})$ and $E(H) = E(I_{1, t-2}) \cup \{((1,t-4), (2,t-4))\}$. By repeating this process for all pairs $(i, j),\  1 \leq i \neq j \leq n$, we get the graph $\tilde{I}_{1,t-2}$ such that $\ind(I_{1,t-2}) \simeq \ind(\tilde{I}_{1,t-2})$, where $V(\tilde{I}_{1,t-2}) = V(I_{1,t-2})$ and $E(\tilde{I}_{1,t-2}) =E(I_{1,t-2}) \cup \{((i,t-4), (j,t-4))\ : \  1 \leq i \neq  j \leq n\}$. Since $N((1,t-3)) = \{(2,t-4), \ldots, (n,t-4)\}$, the vertex $(1,t-3)$ is a simplicial vertex in $\tilde{I}_{1,t-2}$ .  From Lemma \ref{lem:simplicial}, we have 
	\begin{center}
	$\ind(I_{1,t-2}) \simeq \ind(\tilde{I}_{1,t-2}) \simeq \bigvee\limits_{i= 2}^{n} \Sigma \big{(}\ind(\tilde{I}_{1,t-2} - N[	\{(i,t-4)\}])\big{)}$.
	\end{center}

Observe that for each $2 \leq i \leq n$,  the graph $\tilde{I}_{1,t-2} - N[	\{(i,t-4)\}]$ is isomorphic to $I_{i, t-5} \sqcup K_2$, where $K_2$ appears because of the edge $((1, t-2), (i, t-3))$ in $\tilde{I}_{1,t-2}- N[\{(i,t-4)\}]$. We note that for any $j$,  $I_{i, j} \cong I_{l, j}$  and therefore 
	\begin{align}\label{eq:I{1,t-2}}
	\textstyle \ind(I_{1,t-2})
	\simeq \bigvee\limits_{(n-1)} \Sigma \Big( \ind(I_{1, t-5} \sqcup K_2)\Big) \simeq \bigvee\limits_{(n-1)} \Sigma^2 \Big( \ind(I_{1, t-5}\big)\Big).
	\end{align}

We consider the following three cases.

\noindent{\bf Case 1.} $t = 3k$.

In this case, since $t-2= 3(k-1)+1$, using  Equation \eqref{eq:I{1,t-2}} and Lemma \ref{lem:base}, we conclude that 
\begin{align*}
\textstyle \ind(I_{1,t-2}) \simeq \bigvee\limits_{(n-1)^{k-1}} \Sigma^{2(k-1)} (\ind(I_{1,1}) )\simeq \bigvee\limits_{(n-1)^{k-1}} \Sigma^{2(k-1)} (\bigvee\limits_{n-1} \bS^{0}) \simeq \bigvee\limits_{(n-1)^{k}} \bS^{2(k-1)}.
\end{align*}

\noindent {\bf Case 2.} $t = 3k+1$.

	In this case, $t-2= 3(k-1)+2$. Again by Lemma \ref{lem:base} and Equation \eqref{eq:I{1,t-2}}, we have 
$$\textstyle \ind(I_{1,t-2}) \simeq \bigvee\limits_{(n-1)^{k-1}} \Sigma^{2(k-1)} (\ind(I_{1,2}) )\simeq \bigvee\limits_{(n-1)^{k-1}} \Sigma^{2(k-1)} (\bigvee\limits_{n-1} \bS^{1}) \simeq \bigvee\limits_{(n-1)^{k}} \bS^{2(k-1)+1}.
$$

\noindent {\bf Case 3.} $t = 3k+2$.

In this case, since $t-2= 3(k-1)+3$, 
$\ind(I_{1,t-2}) \simeq \bigvee\limits_{(n-1)^{k-1}} \Sigma^{2(k-1)} (\ind(I_{1,3}) )$. By Lemma \ref{lem:base}, $\ind(I_{1,3})$ is contractible and so is  $\ind(I_{1,t-2})$.
\end{proof}

We are now ready to prove the main result of this section. Firstly, note that $M_r(K_2)$ is isomorphic to odd cycle $C_{2r+1}$, for which independence complex has been computed by Kozlov in \cite{koz}. Here, we determine the homotopy type of $\ind(M_r(K_n))$ for $n > 2$ and any $r$.

 \begin{thm} \label{thm:mycielskian} Let $r \geq 2$ be a positive integer. Then,
	
	\begin{center}
		$ \ind(M_r(K_n)) \simeq 
		\begin{cases}
		\bigvee\limits_{(n-1)^k} \bS^{2k-1} & \text{if } r=3k, \\
		\bigvee\limits_{n(n-1)^k} \bS^{2k} & \text{if } r=3k+1, \\
		\bigvee\limits_{(n-1)^{(k+1)}} \bS^{2k+1} & \text{if } r= 3k+2. \\
		\end{cases}$
	\end{center}
\end{thm}
\begin{proof}
	
	If $r =2$, then  Theorem \ref{thm:requal2} implies that $\ind(M_r(K_n)) \simeq \Sigma (\ind(K_n))$, and hence the result follows. So we  assume that  $r \geq 3$.  
	In $M_r(K_n)$, $N(w_r) = \{(1, r-1), \ldots, (n, r-1)\}$ and there is no edge among the vertices of $N(w_r)$. Therefore, from Lemma \ref{lem:cluster} 
	
	\begin{align}\label{eqn:cluster}
	\ind(M_r(K_n)) \simeq \Sigma \big( st_{\ind(M_r(K_n))}(w_r) \cap SC_{\ind(M_r(K_n))}(\{(1,r-1), \ldots, (n, r-1)\})\big).
\end{align}

\noindent	Let $K$ denote the complex $ st_{\ind(M_r(K_n))}(w_r) \cap SC_{\ind(M_r(K_n))}(\{(1,r-1), \ldots, (n, r-1)\})$.

	\begin{claim}\label{claim:setK}
		$K = \bigcup\limits_{i=1}^{n}\emph{\ind}(I_{i,r-2}).$
	\end{claim}
		Let $\sigma \in K$. If $\sigma \cap \{(1, r-2), \ldots, (n, r-2)\} = \emptyset$, then $\sigma \in \bigcup\limits_{i=1}^{n}\ind(I_{i,r-2})$. On the other hand  if $ \sigma \cap \{(1, r-2), \ldots, (n, r-2)\} \neq \emptyset$, then there exists a unique $i$ such that $(i, r-2) \in \sigma$ and in this case $\sigma \in \ind(I_{i, r-2})$. Hence, $K \subseteq \bigcup\limits_{i=1}^{n}\ind(I_{i,r-2})$. If $\tau \in \bigcup\limits_{i=1}^{n}\ind(I_{i,r-2})$, then $\tau \in \ind(I_{i, r-2})$ for some $i$, therefore  $\tau \in st_{\ind(M_r(K_n))}((i, r-1)) \cap st_{\ind(M_r(K_n))}(w_r)$. This completes the proof of Claim \ref{claim:setK}.
	
    	For $ 1\leq i\neq l \leq n$, observe that $I_{i,j} \cap I_{l,j} \cong K_n \times L_{j-1}$, therefore $\ind(I_{i,j} \cap I_{l,j}) \simeq \ind(K_n \times L_{j-1})$. Also, for any arbitrary graph $G$ and $A,B \subseteq V(G)$, $\ind(G[A\cap B]) = \ind(G[A]) \cap \ind(G[B])$, and hence   $\ind(I_{i,j} \cap I_{l, j}) =  \ind(I_{i,j}) \cap \ind(I_{l, j})$. Further, $\big(\bigcup\limits_{s= 1}^{i} \ind(I_{s, j})\big) \bigcap \ind(I_{i+1, j}) = \ind(I_{i, j})\cap \ind(I_{i+1,j})\simeq \ind(K_n \times L_{j-1})$ for all $1 \leq i \leq n-1$.\\

\noindent{\bf Case 1:} $r = 3k+1$.

From Lemmas  \ref{lem:base} and \ref{lem:I_{i,r-2}}, $\ind(I_{i,r-2}) \simeq \bigvee\limits_{(n-1)^{k}} \bS^{2(k-1)+1}$.  Proposition  \ref{prop:main} implies that  $\ind(K_n \times L_{r-3})$ is contractible and therefore  by using Claim \ref{claim:setK} and Lemma \ref{lem:wedge}, we conclude that $K \simeq \bigvee\limits_{n(n-1)^k} \bS^{2(k-1)+1}$. Thus,  from Equation \eqref{eqn:cluster}, $$\textstyle \ind(M_r(K_n)) \simeq \Sigma(K) \simeq \bigvee\limits_{n(n-1)^k} \bS^{2k}.$$

\noindent{\bf Case 2 :} $r  = 3k+2$.

In this case, $\ind(I_{i,r-2})$ is contractible  from Lemmas \ref{lem:base} and \ref{lem:I_{i,r-2}}. Since,  $r-3 = 3(k-1)+2$, $\ind(K_n \times L_{r-3})  \simeq \bigvee\limits_{(n-1)^k} \bS^{2k-1}$ from Lemma \ref{prop:main}. Hence by using Lemma \ref{lem:contractible} and Claim \ref{claim:setK}, we conclude that  $K \simeq \bigvee\limits_{(n-1)(n-1)^k} \bS^{2k} $. Thus,
$$\ind(M_r(K_n)) \simeq \bigvee\limits_{(n-1)^{k+1}} \bS^{2k+1}.$$

\noindent{\bf Case 3.} $r = 3k$.

First we show that $K \simeq \ind(K_n \times L_{r-3})$. Observe that $lk_K((1,r-2)) = \ind(I_{1,r-3})$. Using Lemmas \ref{lem:base} and \ref{lem:I_{i,r-2}}, we get that $\ind(I_{1,r-3})$ is contractible. Hence  Lemma \ref{lem:link} implies that $K \simeq K^{1} := K \setminus \{\sigma \in K \ : \ (1,r-3)\in \sigma\}$.  Now  $lk_{K^1}((2,r-2)) = \ind(I_{2,r-3})\cong \ind(I_{1,r-3})$ and therefore $K^1 \simeq K^2 := K^1 \setminus \{\sigma \in K^1  :  (2,r-3) \in \sigma\}$. Repeating this process for all $(i, r-2), 3 \leq i \leq n$, we conclude that $K \simeq K^n  :=\{\sigma \in K \ : \ (1,r-2), \ldots, (n,r-2) \notin \sigma\}$. It is easy to check  that 
$K^n \cong \ind(K_n \times L_{r-3})$.
From Proposition \ref{prop:main}, we conclude that $$\ind(M_r(K_n))  \simeq \Sigma(\ind(K_n \times L_{r-3})) \simeq \Sigma(\bigvee\limits_{(n-1)^k} \bS^{2(k-1)}) \simeq \bigvee\limits_{(n-1)^k} \bS^{2k-1},$$ and hence the theorem.
\end{proof}


\section{Independence complexes of graphs with a specific local structure}\label{sec:subdivision}

Let $G$ be a graph that contains a crossing of two edge as in Figure \ref{crossing}. In this section we prove that if we replace  a crossing in graph $G$ with the structure as in Figure \ref{ladder}, then the independence complex of  $H$ is the suspension of the independence complex of $G$. 

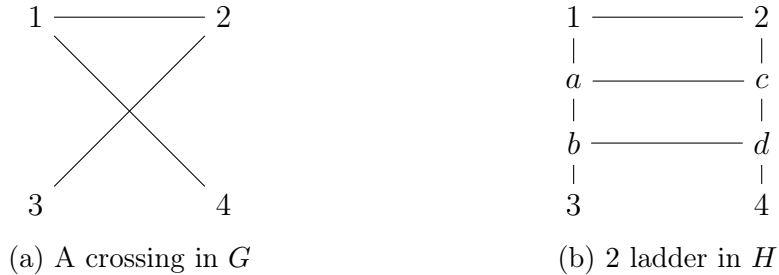
\begin{figure}[H]
	\begin{subfigure}[]{0.45 \textwidth}
		\centering
		\begin{tikzpicture}
		[scale=0.5, vertices/.style={circle, draw = black!100, fill= gray!0, inner sep=0.5pt}] 
		\node (a) at (5,0)  {$4$};
		\node (b) at (5,5) {$2$};
		\node (c) at (0,5)  {$1$};
		\node (d) at (0,0)  {$3$};
		
		\foreach \from/\to in {a/c, b/d, b/c}
		\draw (\from) -- (\to);
		\end{tikzpicture}
		\caption{A crossing in $G$}\label{crossing}
	\end{subfigure}
	\begin{subfigure}[]{0.45 \textwidth}
		\centering
		\begin{tikzpicture}
		[scale=0.5, vertices/.style={circle, draw = black!100, fill= gray!0, inner sep=0.5pt}]
		\node (a) at (5,0)  {$4$};
		\node (b) at (5,5) {$2$};
		\node (c) at (0,5)  {$1$};
		\node (d) at (0,0)  {$3$};
		\node (e) at (0,1.65) {$b$};
		\node (f) at (0,3.30) {$a$};
		\node (g) at (5,1.65) {$d$};
		\node (h) at (5,3.30) {$c$};
		
		\foreach \from/\to in {b/c, c/f, f/e, f/h, e/g, e/d, b/h, h/g, a/g}
		\draw (\from) -- (\to);
		\end{tikzpicture}
		\caption{$2$ ladder in $H$}\label{ladder}
	\end{subfigure}
	\caption{Replacing a crossing in $G$ by a 2 ladder to obtain $H$} \label{fig:subdivision}
\end{figure}

\begin{thm}\label{thm:subdivision}
	Let $G$ be a graph that contains graph depicted in Figure \ref{crossing} as a subgraph. If  $H$ is the graph with $V(H) = V(G) \sqcup \{a,b,c,d\}$, $E(H) = (E(G) \setminus \{(1,4), (2,3)\}) \cup \{(1,a), (a,b), (b,3), (2,c), (c,d), (d,4), (a,c), (b,d)\}$  obtained from $G$ as shown in Figure \ref{ladder}, then    
	$$\emph{\ind(H)}\simeq \Sigma(\emph{\ind(G)}).$$ 
\end{thm}

\begin{proof}
Observe that $\{a,d\}$ and $\{b,c\}$ are simplices of Ind$(H)$. Let $K=\text{Ind}(H)$, $K_1=SC_{K}(\{a,d\})$ and $K_2=SC_{K}(\{b,c\})$. From Lemma \ref{lem2.2}, $K_1$ and $K_2$ are contractible subcomplexes of $K$. 

We  first show that $K=K_1 \cup K_2.$ Clearly, $K_1 \cup K_2$ is a subcomplex of $K$. Let $\sigma \in K$. If $\{a,b,c,d\} \cap \sigma \neq \emptyset$, then by definition $\sigma \in K_1$ or $K_2$. Therefore we assume that  $\{a,b,c,d\} \cap \sigma = \emptyset$. If $1 \notin \sigma$, then $\sigma \in st_K(\{a\})$. If $1 \in \sigma$, then $2 \notin \sigma$ and thereby implying that $\sigma \in st_K(\{c\})$. 

\begin{claim}\label{claim2}
$K_1 \cap K_2=\ind(G).$
\end{claim}
Let $\sigma \in K_1\cap K_2$. Clearly, $\{a,b,c,d\} \cap \sigma =\emptyset$. To show that $\sigma \in \text{Ind}(G)$, it is enough to show that $\{1,4\}\nsubseteq \sigma$ and $\{2,3\} \nsubseteq \sigma$. However, $\{1,4\} \notin K_1$ implies that $\{1,4\}\nsubseteq \sigma$. Similarly, $\{2,3\}\notin K_2$ implies that $\{2,3\} \nsubseteq \sigma$. 

Now, let $\sigma \in \text{Ind}(G)$. Since $\{1,4\}\nsubseteq \sigma$, $\sigma$ is in $st_K(\{a\})$ or $st_K(\{d\})$. Moreover, $\{2,3\}\nsubseteq \sigma$ implies that either $\sigma$ is in $st_K(\{b\})$ or is in $st_K(\{c\})$. Therefore, $\sigma \in K_1 \cap K_2.$

Thus result follows from Lemma \ref{lem:suspension}.
\end{proof}

    We note that the proof of Theorem \ref{thm:subdivision} also holds if we assume that the vertices $3$ and $4$ are same (cf. Figure \ref{fig:subdivision2}), {\it i.e.}, $3=4$. We, therefore, have the following result as a special case.

\begin{thm}[Section $3.3.1$, \cite{JS10}]\label{thm:subdivision1}
	Let $G$ be a graph that contains triangle  depicted in Figure \ref{triangle} as a subgraph. If  $H$ is the graph with $V(H) = V(G) \sqcup \{a,b,c,d\}$, $E(H) = (E(G) \setminus \{(1,3), (2,3)\}) \cup \{(1,a), (a,b), (b,3), (2,c), (c,d), (d,3), (a,c), (b,d)\}$  obtained from $G$ as shown in Figure \ref{ladder1}, then    
	$$\emph{\ind(H)}\simeq \Sigma(\emph{\ind(G)}).$$ 
\end{thm}

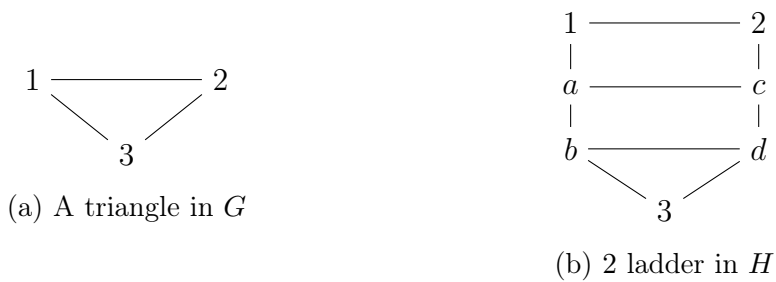
\begin{figure}[H]
	\begin{subfigure}[]{0.45 \textwidth}
		\centering
		\begin{tikzpicture}
		[scale=0.5, vertices/.style={circle, draw = black!100, fill= gray!0, inner sep=0.5pt}] 
		\node (a) at (2.5,3)  {$3$};
		\node (b) at (5,5) {$2$};
		\node (c) at (0,5)  {$1$};
		
		\foreach \from/\to in {a/c,  b/c, a/b}
		\draw (\from) -- (\to);
		\end{tikzpicture}
		\caption{A triangle in $G$}\label{triangle}
	\end{subfigure}
	\begin{subfigure}[]{0.45 \textwidth}
		\centering
		\begin{tikzpicture}
		[scale=0.5, vertices/.style={circle, draw = black!100, fill= gray!0, inner sep=0.5pt}]
		\node (a) at (2.5,0)  {$3$};
		\node (b) at (5,5) {$2$};
		\node (c) at (0,5)  {$1$};
		\node (e) at (0,1.65) {$b$};
		\node (f) at (0,3.30) {$a$};
		\node (g) at (5,1.65) {$d$};
		\node (h) at (5,3.30) {$c$};
		
		\foreach \from/\to in {a/e,b/c, c/f, f/e, f/h, e/g,  b/h, h/g, a/g}
		\draw (\from) -- (\to);
		\end{tikzpicture}
		\caption{$2$ ladder in $H$} \label{ladder1}
	\end{subfigure}
	\caption{Replacing an edge in $G$ by a 2 ladder to obtain $H$} \label{fig:subdivision2}
\end{figure}

Before proceeding further, we would like to point out that Skwarski has considered a similar construction as in Figure \ref{fig:subdivision2} in \cite[Section 3.3.1]{JS10}.

We now record a straight forward observation that follows from Theorem \ref{thm:subdivision1}.

\begin{cor} Let $G$ be a graph with Figure \ref{ladder1} as an induced subgraph. Then the independence complex of $G$ has the homotopy type of a suspension.
\end{cor}

As an application of Theorem \ref{thm:subdivision} and Theorem \ref{thm:subdivision1}, we compute the  homotopy type of the independence complexes of a particular family of graphs.

	Let $C_n^0 \equiv C_n$ be the cycle graph on the vertex set $\{1,2,\ldots,n\}$.
	Let $C_n^1$ be the graph obtained from $C_n$ by subdividing the edges adjacent to $1$ and adding an edge between the newly created  vertices. Let $x_1, y_1$ be the vertices of $V(C_n^1) \setminus V(C_n^0)$. We iteratively define the graph $C_n^j$ to be the graph obtained from $C_n^{j-1}$ as per the above construction. We note that $V(C_n^j) = \{1,2, \dots, n\} \cup \{x_1, x_2, \dots, x_j\} \cup \{y_1, y_2, \dots, y_j\}$, {\it i.e.}, $|V(C_n^j)| = n + 2j$ and
	$	E(C_n^j)  =  (E(C_n^{0}) \setminus \{(1,2), (1,n)\}) 
	 \cup \{(x_i, y_i)  :  1 \leq i \leq j\}
	\cup \{(x_i, x_{i+1}),(y_i, y_{i+1}) : 1 \leq i \leq j-1\} 
	\cup \{(1, x_j), (1, y_j), (2,x_1), (n, y_1)\}$.
	
	For example, Figure \ref{subcycles} shows $C_3^2$ and $C_5^3$.
	
\begin{figure}[H]
	\begin{subfigure}[]{0.45 \textwidth}
		\centering
		\begin{tikzpicture}
		[scale=0.4, vertices/.style={circle, draw = black!100, fill= gray!0, inner sep=0.5pt}] 
		\node (1) at (3,9)  {$1$};
		\node (2) at (0,0)  {$2$};
		\node (3) at (6,0)  {$3$};
		\node (4) at (2,6)  {$x_2$};
		\node (5) at (1,3)  {$x_1$};
		\node (6) at (4,6)  {$y_2$};
		\node (7) at (5,3)  {$y_1$};
		
		\foreach \from/\to in {1/4, 1/6, 5/4, 7/6, 5/2, 7/3, 2/3, 4/6, 5/7}
		\draw (\from) -- (\to);
		\end{tikzpicture}
		\caption{$C_3^2$}\label{C32}
	\end{subfigure}
	\begin{subfigure}[]{0.45 \textwidth}
		\centering
		\begin{tikzpicture}
		[scale=0.40, vertices/.style={circle, draw = black!100, fill= gray!0, inner sep=0.5pt}]
		\node (1) at (4,10)  {$1$};
		\node (2) at (0,2)  {$2$};
		\node (3) at (2.5,0)  {$3$};
		\node (4) at (5.5,0)  {$4$};
		\node (5) at (8,2)  {$5$};
		\node (6) at (3,8)  {$x_3$};
		\node (7) at (2,6)  {$x_2$};
		\node (8) at (1,4)  {$x_1$};
		\node (9) at (5,8)  {$y_3$};
		\node (10) at (6,6)  {$y_2$};
		\node (11) at (7,4)  {$y_1$};
		
		\foreach \from/\to in {1/6, 2/3, 3/4, 4/5, 5/11, 1/9, 9/10, 10/11, 6/7, 7/8, 8/2, 6/9, 7/10, 8/11}
		\draw (\from) -- (\to);
		\end{tikzpicture}
		\caption{$C_5^3$}\label{C53}
	\end{subfigure}
	\caption{Subdivision of cycles}\label{subcycles}
\end{figure}
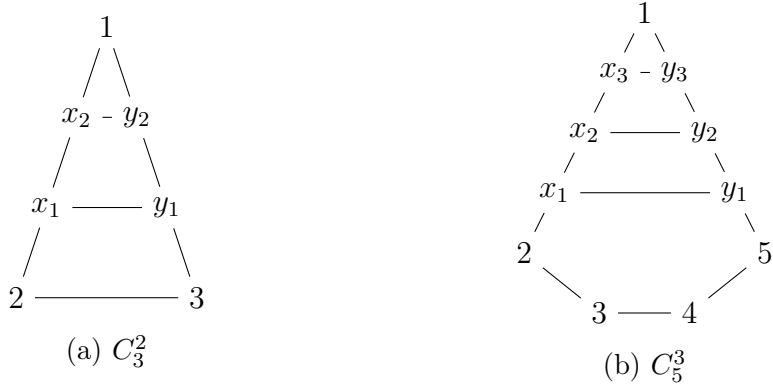

\begin{cor}Let $i,r \geq 1$, then
	\begin{center}
		$ \ind(C_{3r}^i) \simeq 
		\begin{cases}
		\bS^{r-1+j} \bigvee \bS^{r-1+j} & \text{if } i=2j, \\
		\text{\rm \{\text{point}\}} & \text{otherwise}, \\
		\end{cases}$
	\end{center}
	\begin{center}
		$ \ind(C_{3r+1}^i) \simeq 
		\begin{cases}
		\bS^{r-1+j} \bigvee \bS^{r-1+j} & \text{if } i=2j-1, \\
		\text{\rm \{point\}} & \text{otherwise}, \\
		\end{cases}$
	\end{center}
	and
	\begin{center}
		$ \ind(C_{3r+2}^i) \simeq 
		\begin{cases}
		\bS^{r} \bigvee \bS^{r} & \text{if } i=1, \\
		\bS^{r+j} \bigvee \bS^{r+j} & \text{if } i=2j \text{ or } 2j+1, \ j \geq 1. \\
		\end{cases}$
	\end{center}
\end{cor}

\begin{proof}
    Let $P_n$ be the path graph on $n$ vertices with $n-1$ edges.  From \cite[Proposition 11.16]{dk}, we know that
     \begin{equation} \label{eq:path}
		 \ind(P_n) \simeq 
		\begin{cases}
		\bS^{r-1} & \text{if } n=3r, \\
		\{\text{point}\} & \text{if } n=3r+1, \\
		\bS^{r} & \text{if } 3r+2. 
		\end{cases}
	\end{equation}
 We give the proof  by induction on $i$. Observe that in $C_{n}^{i}$, vertex $1$ is a simplicial vertex for each $i \geq 1$. Therefore, using Lemma \ref{lem:simplicial} and Equation \ref{eq:path} we get 
		\begin{align*}
	\ind(C_{n}^{1}) & \simeq \Sigma \big(\ind(C_{n}^{1} -\{1,2,x_1,y_1\})\big) {\textstyle \bigvee} \Sigma \big(\ind(C_{n}^{1} -\{1,n,x_1,y_1\})\big) \\
	& \simeq \Sigma \big(\ind(P_{n-2})\big) {\textstyle \bigvee} \Sigma\big(\ind(P_{n-2})\big) \\
	& \simeq \begin{cases}\{\text{point}\}  & \text{if } n=3r,\\
	\bS^r {\textstyle \bigvee} \bS^r & \text{if } n=3r+ 1,3r+2.
	\end{cases}
	\end{align*}
	
	Similarly, using Lemma \ref{lem:simplicial} and Equation \ref{eq:path} we get
	 \begin{align*}
	\ind(C_{n}^{2}) & \simeq \Sigma \big(\ind(C_{n}^{2} -\{1,x_1,x_2,y_2\})\big) {\textstyle \bigvee} \Sigma\big(\ind(C_{n}^{1} -\{1,x_2,y_1,y_2\})\big) \\
	& \simeq \Sigma\big(\ind(P_{n})\big) {\textstyle \bigvee} \Sigma\big(\ind(P_{n})\big) \\
	& \simeq 
	\begin{cases}
	\bS^r {\textstyle \bigvee} \bS^r & \text{if } n=3r,\\
	\{\text{point}\} & \text{if } n=3r+1, \\
	\bS^{r+1} {\textstyle \bigvee} \bS^{r+1} & \text{if } n=3r+2.
	\end{cases}
	\end{align*}
	
	Now using Theorem \ref{thm:subdivision1}, we observe that for any $i \geq 3$, $\ind(C_{n}^{i}) = \Sigma(\ind(C_{n}^{i-2}))$ and therefore by induction on $i$, the result follows from $\ind(C_n^{1})$ and $\ind(C_n^2)$.
\end{proof}
  \subsection*{Acknowledgements}
  We are thankful to Priyavrat Deshpande and Dheeraj Kulkarni for inviting us to the workshop ``Young Topologists' Meet" held at Chennai Mathematical Institute in 2018, where the work of Section \ref{sec:subdivision} was done. The first author was partially supported by IRCC, IIT Bombay and the third author was partially supported by a grant from Infosys Foundation.

   

\begin{thebibliography}{10}
	
	\bibitem{abz}
	Ron Aharoni, Eli Berger, and Ran Ziv.
	\newblock Independent systems of representatives in weighted graphs.
	\newblock {\em Combinatorica}, 27(3):253--267, 2007.
	
	\bibitem{BK1}
	Eric Babson and Dmitry~N. Kozlov.
	\newblock Proof of the {L}ov\'{a}sz conjecture.
	\newblock {\em Ann. of Math. (2)}, 165(3):965--1007, 2007.
	
	\bibitem{Bar13}
	Jonathan~Ariel Barmak.
	\newblock Star clusters in independence complexes of graphs.
	\newblock {\em Adv. Math.}, 241:33--57, 2013.
	
	\bibitem{BLN}
	Mireille Bousquet-M\'{e}lou, Svante Linusson, and Eran Nevo.
	\newblock On the independence complex of square grids.
	\newblock {\em J. Algebraic Combin.}, 27(4):423--450, 2008.
	
	\bibitem{BB}
	Benjamin Braun.
	\newblock Independence complexes of stable {K}neser graphs.
	\newblock {\em Electron. J. Combin.}, 18(1):Paper 118, 17, 2011.
	
	\bibitem{bredon}
	Glen~E Bredon.
	\newblock {\em Topology and geometry}, volume 139.
	\newblock Springer Science \& Business Media, 2013.
	
	\bibitem{csorba}
	P\'{e}ter Csorba.
	\newblock Subdivision yields {A}lexander duality on independence complexes.
	\newblock {\em Electron. J. Combin.}, 16(2, Special volume in honor of Anders
	Bj\"{o}rner):Research Paper 11, 7, 2009.
	
	\bibitem{eh}
	Richard Ehrenborg and G\'{a}bor Hetyei.
	\newblock The topology of the independence complex.
	\newblock {\em European J. Combin.}, 27(6):906--923, 2006.
	
	\bibitem{ea1}
	Alexander Engstr\"{o}m.
	\newblock Independence complexes of claw-free graphs.
	\newblock {\em European J. Combin.}, 29(1):234--241, 2008.
	
	\bibitem{ea2}
	Alexander Engstr\"{o}m.
	\newblock Complexes of directed trees and independence complexes.
	\newblock {\em Discrete Math.}, 309(10):3299--3309, 2009.
	
	\bibitem{ea}
	Alexander Engstr\"{o}m.
	\newblock A local criterion for {T}verberg graphs.
	\newblock {\em Combinatorica}, 31(3):321--332, 2011.
	
	\bibitem{f}
	Robin Forman.
	\newblock Morse theory for cell complexes.
	\newblock {\em Adv. Math.}, 134(1):90--145, 1998.
	
	\bibitem{kk1}
	Kazuhiro Kawamura.
	\newblock Homotopy types of independence complexes of forests.
	\newblock {\em Contrib. Discrete Math.}, 5(2):67--75, 2010.
	
	\bibitem{kk}
	Kazuhiro Kawamura.
	\newblock Independence complexes of chordal graphs.
	\newblock {\em Discrete Math.}, 310(15-16):2204--2211, 2010.
	
	\bibitem{dk}
	Dmitry Kozlov.
	\newblock {\em Combinatorial algebraic topology}, volume~21 of {\em Algorithms
		and Computation in Mathematics}.
	\newblock Springer, Berlin, 2008.
	
	\bibitem{koz}
	Dmitry~N. Kozlov.
	\newblock Complexes of directed trees.
	\newblock {\em J. Combin. Theory Ser. A}, 88(1):112--122, 1999.
	
	\bibitem{RM}
	Roy Meshulam.
	\newblock Domination numbers and homology.
	\newblock {\em J. Combin. Theory Ser. A}, 102(2):321--330, 2003.
	
	\bibitem{NS}
	Nandini {Nilakantan} and Samir {Shukla}.
	\newblock {Homotopy type of the independence complexes of a family of regular
		bipartite graphs}.
	\newblock {\em arXiv e-prints}, page arXiv:1709.04789, Sep 2017.
	
	\bibitem{JS10}
	Jakob Skwarski.
	\newblock {\em Operations on a graph G that shift the homology of the
		independence complex of G}.
	\newblock Skolan f{\"o}r teknikvetenskap, Kungliga Tekniska h{\"o}gskolan,
	2010.
	
\end{thebibliography}


\end{document}